\documentclass[reqno]{amsart}
\usepackage{amsmath}
\usepackage{amssymb}
\usepackage{hyperref}
\hypersetup{ colorlinks = true, urlcolor = blue, linkcolor = blue, citecolor = red }
\usepackage[margin=2.2cm]{geometry}
\usepackage{xcolor}
\usepackage{comment}

\DeclareMathOperator{\loc}{loc}

\DeclareMathOperator{\osc}{osc}

\newcommand{\R}{\mathbb{R}}
\newcommand{\norm}[1]{\left\lVert#1\right\rVert}
\newcommand{\inner}[2]{\left\langle #1, #2 \right\rangle}

\newcommand{\I}{\int\limits}

\theoremstyle{plain}
\newtheorem{thm}{Theorem}[section]
\newtheorem{lem}{Lemma}[section]

\theoremstyle{definition}
\newtheorem{defn}{Definition}[section]

\theoremstyle{remark}

\usepackage{nameref}
\makeatletter
\let\orgdescriptionlabel\descriptionlabel
\renewcommand*{\descriptionlabel}[1]{%
	\let\orglabel\label
	\let\label\@gobble
	\phantomsection
	\edef\@currentlabel{#1}%
	\let\label\orglabel
	\orgdescriptionlabel{#1}%
}
\makeatother
\numberwithin{equation}{section}

\begin{document}
	
\title[]{$C^{1,\alpha}$-regularity for a class of degenerate/singular fully nonlinear elliptic equations}

\everymath{\displaystyle}

\author{Sumiya Baasandorj}
\address{Department of Mathematical Sciences, Seoul National University, Seoul 08826, Korea}
\email{summa2017@snu.ac.kr}
\thanks{This work was supported by NRF-2021R1A4A1027378.}

\author{Sun-Sig Byun}
\address{Department of Mathematical Sciences and Research Institute of Mathematics,
Seoul National University, Seoul 08826, Korea}
\email{byun@snu.ac.kr}

\author{Ki-Ahm Lee}
\address{Department of Mathematical Sciences and Research Institute of Mathematics,
Seoul National University, Seoul 08826, Korea}
\address{Center for Mathematical Challenges, Korea Institute for Advanced Study, Seoul 02455, Korea}
\email{kiahm@snu.ac.kr}

\author{Se-Chan Lee}
\address{Department of Mathematical Sciences, Seoul National University, Seoul 08826, Korea}
\email{dltpcks1@snu.ac.kr}


\begin{abstract}
	We establish an optimal $C^{1,\alpha}$-regularity for viscosity solutions of degenerate/singular fully nonlinear elliptic equations by finding minimal regularity requirements on the associated operator.
\end{abstract}

\maketitle
\tableofcontents

\section{Introduction}

In this paper we provide a unified way for proving H\"older regularity for the gradient of viscosity solutions to fully nonlinear elliptic equations of the form
\begin{align}
    \label{me}
    \begin{split}
    \Phi(x,|Du|) F(D^{2}u) = f(x)
    \quad\text{in}\quad B_{1},
    \end{split}
\end{align}
where $B_{1}\equiv B_{1}(0)\subset \R^{n}$ with $n\geqslant 2$ is the unit ball, $F : \mathcal{S}(n)\rightarrow \R$ is a uniformly $(\lambda,\Lambda)$-elliptic operator in the sense of \ref{a1} and $\Phi : B_{1}\times [0,\infty)\rightarrow [0,\infty)$ is a continuous map featuring a degeneracy and singularity for the gradient described as in \ref{a2}. From a variational point of view, the fully nonlinear equation \eqref{me} is closely related to the energy functional 
\begin{align}
	\label{func}
	v\mapsto \I_{B_{1}} \varphi(x,|Dv|)\,dx
\end{align}
for a integral density $\varphi : B_{1}\times [0,\infty)\rightarrow [0,\infty)$ in a way that the Euler-Lagrange equation corresposnding to the functional \eqref{func} forms an equation of type \eqref{me}.
 The functional in \eqref{func} is a highly general non-autonomous functional with Uhlenbeck structure including significant models such as $p-$, Orlicz-, $p(x)-$ and double phase- growth and so on. H\"older continuity for the gradient of local minima of the functional \eqref{func} under suitable optimal assumptions has been investigated in \cite{HO}, where fundamental assumptions on the integral density function $\varphi$ in \eqref{func} are that there exist constants $1<p,q$ such that the map $t\mapsto \frac{\varphi(x,t)}{t^{p}}$ is almost non-decreasing and the map $t\mapsto \frac{\varphi(x,t)}{t^{q}}$ is almost non-increasing, see \cite[Definition 3.1]{HO}. In this regard, our conditions on $\Phi$ in \eqref{me} to be introduced in \ref{a2} are absolutely reasonable. Let us present known regularity results for viscosity solutions of equations in the form of \eqref{me} as significant special cases of our problem. 

\begin{enumerate}
	\item[1.] For $\Phi(x,t) = t^{p}$ with $i(\Phi) = s(\Phi)=p>-1$ in condition \ref{a2}, fully nonlinear equations \eqref{me} with this type of $\Phi(x,t)$ have  been studied in a series of papers. The authors of \cite{BD1} proved the comparison principle and Liouville type theorems in the singular case $(-1<p<0)$, and showed the regularity and uniqueness of the first eigenfunction in \cite{BD2}. Alexandrov-Bakelman-Pucci estimates and the Harnack inequality have been also obtained in \cite{DFQ1,DFQ2,Im1}. In particular, the authors of \cite{IS1} proved local H\"older continuity for the gradient of viscosity solutions of \eqref{me} in the degenerate case ($p\geqslant 0$). Moreover, the authors of \cite{ART1} proved the optimality of H\"older regularity for the gradient of viscosity solutions for the same problem in \cite{IS1} by showing that viscosity solutions are $C_{\loc}^{1,\beta}$ with $\beta=\min\left\{\bar{\alpha}, \frac{1}{p+1}\right\}$, where $\bar{\alpha}\in (0,1)$ is the H\"older exponent coming from the Krylov-Safonov regularity for the homogeneous equation $F(D^2h)=0$. 
	\item[2.] For $\Phi(x,t) = t^{p} + a(x)t^{q}$ with $-1<p, q$ and $0\leqslant a(\cdot)\in C(B_{1})$, the constants in \ref{a2} can be determined as $i(\Phi)= \min\{p,q\} $ and $s(\Phi)= \max\{p,q\}$. The author of \cite{De1} proved the local $C^{1,\beta}-$regularity of viscosity solutions of \eqref{me} for $0\leqslant p \leqslant q$. Moreover, in this degenerate case, the sharpness of the local $C^{1,\beta}$-regularity estimates for bounded viscosity solutions is shown in \cite{SR1}. 
	\item[3.] For $\Phi(x,t) = t^{p(x)}$ with $p(\cdot)\in C(B_{1})$, $i(\Phi)=\inf\limits_{x\in B_{1}}p(x)>-1$ and $s(\Phi) = \sup\limits_{x\in B_{1}}p(x)$ in \ref{a2}, $C^{1,\beta}$-regularity of viscosity solutions has been studied in \cite{BPRT1}. In this paper, we provide a novel  way to prove H\"older continuity for the gradient of  viscosity solutions of \eqref{me} for both degenerate/singular cases in the full generality.
	\item[4.] For  $\Phi(x,t)= t^{p(x)} + a(x) t^{q(x)}$ with functions $0\leqslant a(\cdot)\in C(B_{1})$ and $-1< p(\cdot), q(\cdot)$ in $C(B_{1})$, the constants in \ref{a2} are $i(\Phi)=\inf\limits_{x\in B_{1}}\{p(x),q(x)\} $ and $ s(\Phi)= \sup\limits_{x\in B_{1}} \{p(x),q(x)\}$. In \cite{FRZ1}, local H\"older continuity for the gradient has been proved when $0\leqslant p(\cdot)\leqslant q(\cdot)$.  
\end{enumerate}

For a variational point of these special cases we have discussed above, we refer to the recent survey paper \cite{MR1} presenting important results in problems with nonstandard growth conditions. We also point out the very recent paper \cite{Je1} dealing with  viscosity solutions of an equation of the form 
\begin{align}
	\label{eq:va}
	|Du|^{\beta(x,u,Du)}F(D^2u) = f(x)\quad\text{in}\quad B_{1},
\end{align}
where $\beta :  B_{1}\times\R\times \R^{n}\rightarrow \R$ is a map satisfying $0<\beta
_{m}\leqslant \beta(\cdot)\leqslant \beta_{M}$ for some positive constants $\beta_{m}$ and $\beta_{M}$. In \cite{Je1},  local H\"older continuity for the gradient of viscosity solutions of \eqref{eq:va} is obtained under general conditions on the exponent function $\beta(\cdot)$ for the degenerate case, while the singular case is not be treated due to the methods employed there and the equation \eqref{me} can not be represented as \eqref{eq:va} in general.

Finally, let us recall a consequence of the classical Krylov-Safonov Harnack inequality, see \cite{CC1}, that viscosity solutions to the homogeneous equation 
\begin{align}
	\label{eq:hom}
	F(D^2h) = 0 \quad\text{in}\quad B_{1},
\end{align}
under the assumption that $F : \mathcal{S}(n)\rightarrow \R$ satisfies \ref{a1}, are locally of class $C^{1,\bar{\alpha}}(B_{1})$ for a universal constant $\bar{\alpha}\equiv \bar{\alpha}(n,\lambda,\Lambda)\in (0,1)$ with the estimate 
\begin{align}
	\label{alpha}
	\norm{h}_{C^{1,\bar{\alpha}}(B_{1/2})} \leqslant c \norm{h}_{L^{\infty}(B_{1})}
\end{align}
for some constant $c\equiv c(n,\lambda,\Lambda)$.
The main results of this paper read as follows.

\begin{thm}[H\"older continuity of the gradient]
	\label{thm:mthm}
	Let $u\in C(B_{1})$ be a viscosity solution of \eqref{me} under the assumptions \ref{a1}-\ref{a3}. Then $u\in C^{1,\beta}_{\loc}(B_{1})$ for all $\beta>0$ satisfying 
	\begin{align}
		\label{thm:mthm:1}
		\beta < 
		\begin{cases} \min\left\{\bar{\alpha}, \frac{1}{1+ s(\Phi)}\right\} & \mbox{if } i(\Phi)\geqslant 0,\\ 
		\min\left\{\bar{\alpha}, \frac{1}{1+ s(\Phi)-i(\Phi)}\right\} & \mbox{if } -1<i(\Phi)<0, \end{cases}
	\end{align}
	where $\bar{\alpha}$ is given in \eqref{alpha}. Moreover, for every $\beta$ in \eqref{thm:mthm:1}, there exists a constant $c\equiv c(n,\lambda,\Lambda,i(\Phi),L,\beta)$ such that
	\begin{align}
		\label{thm:mthm:2}
		\norm{u}_{L^{\infty}(B_{1/2})} + \sup\limits_{x\neq y\in B_{1/2}} \frac{|Du(x)-Du(y)|}{|x-y|^{\beta}} \leqslant
		c\left( 1 + \norm{u}_{L^{\infty}(B_{1})} + \norm{\frac{f}{\nu_0}}_{L^{\infty}(B_{1})}^{\frac{1}{1+ i(\Phi)}} \right).
	\end{align}
\end{thm}

The results of Theorem \ref{thm:mthm} are sharp in the view of an example given in \cite{IS1}. As we have discussed above, the results of Theorem \ref{thm:mthm} cover the main results of the papers \cite{BPRT1, FRZ1, De1, IS1} for both cases involving degenerate/singular terms in a unified way. Moreover, the results of Theorem \ref{thm:mthm} cover another important cases such as 

\begin{enumerate}
	\item[1.] $\Phi(x,t)=t^{p} + a(x)t^{p}\log(e+t)$ with $-1<p$ and $0\leqslant a(\cdot)\in C(B_{1})$, where the constants in \ref{a2} are given by $i(\Phi)=p$ and $s(\Phi)=p+\varepsilon$ for any $\varepsilon>0$,
	\item[2.] $\Phi(x,t) = \phi(t) + a(x)\psi(t) $ for suitable $N$-functions $\phi$, $\psi$ and $0\leqslant a(\cdot)\in C(B_{1})$.
\end{enumerate}
We again refer to \cite{MR1} for related variational problems.

Finally, we outline the organization of the paper. In the next section we provide basic notations and assumptions to be used, and also smallness regime and basic regularity results. In Section \ref{sec3}, we prove basic regularity properties of viscosity solutions of \eqref{xi-eq} depending on $i(\Phi)$ and the size of the quantity $|\xi|$. Section \ref{sec4} is devoted to the approximation procedure for viscosity solutions of \eqref{me}. Finally, in last section we provide the proof of our main Theorem \ref{thm:mthm}.


\section{Preliminaries}
\label{sec2}
\subsection{Notions and assumptions}
\label{sec2-1}
Throughout the paper, we denote by $B_{r}(x_0):= \{x\in \R^{n} : |x-x_0|< r\}$ the open ball of $\R^{n}$ with $n\geqslant 2$ centered at $x_0$ with positive radius $r$. If the center is clear in the context, we shall omit the center point by writing $B_{r}\equiv B_{r}(x_0)$. Also $B_{1}\equiv B_{1}(0)\subset \R^{n}$ denote the unit ball. We shall always denote by $c$ a generic positive constant, possible varying line to line, having dependecies on parameters using brackets, that is, for example $c\equiv c(n,i(\Phi),\nu_0)$ means that $c$ depends only on $n,i(\Phi)$ and $\nu_0$. For a measurable map $g : \mathcal{B}\subset B_{1}\rightarrow \R^{N}$ $(N\geqslant 1)$ with $\beta\in (0,1]$ being a given number, we shall use the notation
\begin{align*}
	[g]_{C^{0,\beta}(\mathcal{B})}:= \sup\limits_{x\neq y\in \mathcal{B}} \frac{|g(x)-g(y)|}{|x-y|^{\beta}},\quad
	[g]_{C^{0,\beta}}:= [g]_{C^{0,\beta}(B_{1})}.
\end{align*}

Now we state the main assumptions in the paper.

\begin{description}
	\item[(A1)\label{a1}] The operator $F : \mathcal{S}(n)\rightarrow \R$ in \eqref{me} is continuous and uniformly $(\lambda,\Lambda)$-elliptic in the sense that
	\begin{align*}
		\lambda \text{tr}(N) \leqslant F(M)-F(M+N) \leqslant \Lambda \text{tr}(N)
	\end{align*}
	holds with some constants $0<\lambda\leqslant \Lambda$, whenever $M,N\in \mathcal{S}(n)$ with $N\geqslant 0$, where we denote by $\mathcal{S}(n)$ to mean the set of $n\times n$ real symmetric matrices.
	\item[(A2)\label{a2}] $\Phi : B_{1}\times [0,\infty)\rightarrow [0,\infty) $ is a continuous map satisfying the following properties: 
	\begin{enumerate}
		\item[1.] There exist constants $ s(\Phi)\geqslant i(\Phi)>-1$ such that the map $\displaystyle t\mapsto \frac{\Phi(x,t)}{t^{i(\Phi)}}$ is almost non-decreasing with constant $L\geqslant 1$ in $(0,\infty)$ and the map $\displaystyle t\mapsto \frac{\Phi(x,t)}{t^{s(\Phi)}}$ is almost non-increasing with constant $L\geqslant 1$ in $(0,\infty)$ for all $x\in B_{1}$.
		\item[2.] There exists constants $0<\nu_0\leqslant \nu_1$ such that $\displaystyle \nu_{0} \leqslant  \Phi(x,1) \leqslant \nu_{1}$ for all $x\in B_{1}$.
	\end{enumerate}
	\item[(A3)\label{a3}] The term $f$ on the right hand side of \eqref{me} belongs to $C(B_{1})\cap L^{\infty}(B_{1})$.
\end{description}

The Pucci extremal operators $P_{\lambda,\Lambda}^{\pm} : \mathcal{S}(n)\rightarrow \R$ are defined as 
\begin{align*}
	P_{\lambda,\Lambda}^{+}(M):= -\lambda\sum\limits_{\lambda_{k}>0} \lambda_{k} - \Lambda\sum\limits_{\lambda_{k}<0}\lambda_{k}
\end{align*}
and
\begin{align*}
	P_{\lambda,\Lambda}^{-}(M):= -\Lambda\sum\limits_{\lambda_{k}>0} \lambda_{k} -\lambda\sum\limits_{\lambda_{k}<0}\lambda_{k},
\end{align*}
where $\{\lambda_{k}\}_{k=1}^{n}$ are the eigenvalues of the matrix $M$. The $(\lambda,\Lambda)$-ellipticity of the operator $F$ via the Pucci extremal operators can be formulated as 
\begin{align*}
	P_{\lambda,\Lambda}^{-}(N) \leqslant F(M+N)-F(M) \leqslant P_{\lambda,\Lambda}^{+}(N)
\end{align*}
for all $M,N\in \mathcal{S}(n)$.

In what follows, for any vector $\xi\in\R^{n}$, we define a map $G_{\xi} : B_{1}\times \R^{n}\times \mathcal{S}(n)\rightarrow \R$ by
\begin{align}
	\label{g-fun}
	G_{\xi}(x,p,M): = \Phi(x,|\xi+p|)F(M)-f(x)
\end{align}
under the assumptions prescribed in \ref{a1}-\ref{a3}.
Then we shall focus on viscosity solutions of the equation
\begin{align}
	\label{xi-eq}
	G_{\xi}(x,Du,D^2u)=0 \text{ in } B_{1}.
\end{align}

Now we give the definition of a viscosity solution $u$ of the equation \eqref{xi-eq} as follows.

\begin{defn}
	A lower semicontinuous function $v$ is called a viscosity supersolution of \eqref{xi-eq} if for all $x_0\in B_{1}$ and $\varphi\in C^{2}(B_{1})$ such that $v-\varphi$ has a local minimum at $x_0$ and $D\varphi(x_0)\neq 0$, then 
	\begin{align*}
		G_{\xi}(x_{0},D\varphi(x_0),D^{2}\varphi(x_0))\geqslant 0.
	\end{align*}
	An upper semicontinuous function $w$ is called is a viscosity subsolution of \eqref{xi-eq} if for all $x_0\in B_{1}$ and $\varphi\in C^{2}(B_{1})$ such that $w-\varphi$ has a local maximum at $x_0$ and $D\varphi(x_0)\neq 0$, there holds 
	\begin{align*}
		G_{\xi}(x_{0},D\varphi(x_0),D^{2}\varphi(x_0))\leqslant 0.
	\end{align*}
	We say that $u\in C(B_{1})$ is a viscosity solution of \eqref{xi-eq} if $u$ is a viscosity supersolution and a subsolution simultaneously.
\end{defn}

Also we recall a concept of superjet and subjet introduced in \cite{CIL1}.
\begin{defn}
	Let $v: B_{1}\rightarrow \R$ be an upper semicontinuous function and $w:B_{1}\rightarrow \R$ be a lower semicontinuous function.
	\begin{enumerate}
		\item[1.] A couple $(p,M)\in \R^{n}\times\mathcal{S}(n)$ is a superjet of $v$ at $x\in B_{1}$ if 
		\begin{align*}
			v(x+y) \leqslant v(x) + \inner{p}{y} + \frac{1}{2}\inner{My}{y} + O(|y|^2).
		\end{align*}
		\item[2.] A couple $(p,M)\in \R^{n}\times \mathcal{S}(n)$ is a subjet of $w$ at $x\in B_{1}$ if 
		\begin{align*}
			w(x+y)\geqslant w(x) + \inner{p}{y} + \frac{1}{2}\inner{My}{y} + O(|y|^2).
		\end{align*}
		\item[3.] A couple $(p,M)\in \R^{n}\times \mathcal{S}(n)$ is a limiting superjet of $v$ at $x\in B_{1}$ if there exists a sequence $\{x_k,p_{k},M_{k}\}\rightarrow \{x,p,M\}$ as $k\rightarrow \infty$  in a such way that $\{p_{k},M_{k}\}$ is a superjet of $v$ at $x_{k}$ and $\lim\limits_{k\to\infty}v(x_{k}) = v(x)$.
		\item[4.] A couple $(p,M)\in \R^{n}\times \mathcal{S}(n)$ is a limiting subjet of $w$ at $x\in B_{1}$ if there exists a sequence $\{x_k,p_{k},M_{k}\}\rightarrow \{x,p,M\}$ as $k\rightarrow \infty$  in such a way that $\{p_{k},M_{k}\}$ is a subjet of $v$ at $x_{k}$ and $\lim\limits_{k\to\infty}w(x_{k}) = w(x)$.
	\end{enumerate}
\end{defn}

\subsection{Small regime}
\label{sec2-2}

Here we verify that, for a viscosity solution $u$ of \eqref{xi-eq}, we are able to assume 
\begin{align}
	\label{small}
	\osc\limits_{B_{1}} u \leqslant 1
	\quad\text{and}\quad
	\norm{f}_{L^{\infty}(B_{1})} \leqslant \varepsilon_{0}
\end{align}
for some constant $0<\varepsilon_{0}<1$ small enough, and also $\nu_0=\nu_1=1$ without loss of generality.
In order to consider the problem in a small regime as in \eqref{small}, for a fixed ball $B_{R}(x_0)\subset B_{1}$, we define $\bar{u} : B_{1}\rightarrow \R$ by
\begin{align}
	\label{sr:1}
	\bar{u}(x):= \frac{u(x_0+Rx)}{K}
\end{align}
for positive constants $K\geqslant 1\geqslant R$ to be determined later. It can be seen that $\bar{u}$ is a viscosity solution of 
\begin{align}
	\label{sr:2}
	\bar{G}_{\bar{\xi}}(x,D\bar{u},D^2\bar{u}):= \bar{\Phi}(x,|\bar{\xi}+D\bar{u}|)\bar{F}(D^{2}\bar{u})-\bar{f}(x) = 0,
\end{align}
where 
\begin{align*}
	\displaystyle
	\bar{\Phi}(x,t) &:= \frac{\Phi\left(x_0+Rx,\frac{K}{R}t\right)}{\Phi\left(x_0+Rx,\frac{K}{R}\right)}, \\
	\bar{F}(M)&:= \frac{R^{2}}{K}F\left(\frac{K}{R^{2}}M \right),
	\\
	\bar{f}(x) &:= \frac{R^{2}}{\Phi\left(x_0+Rx,\frac{K}{R}\right)K}f(x_0+Rx) \text{ and } \bar{\xi}:= \frac{R}{K}\xi.
\end{align*}

Note that $\bar{F}$ is still a uniformly $(\lambda,\Lambda)$-elliptic operator, the map $\displaystyle t\mapsto \frac{\bar{\Phi}(x,t)}{t^{i(\Phi)}}$ is almost non-decreasing and the map $\displaystyle t\mapsto \frac{\bar{\Phi}(x,t)}{t^{s(\Phi)}}$ is almost non-increasing with the same constants $L\geqslant 1$ and $s(\Phi)\geqslant i(\Phi)>-1$ as in \ref{a2}, and $\bar{\Phi}(x,1)=1$ for all $x\in B_{1}$. Moreover, the assumption \ref{a2} implies
\begin{align*}
	\norm{\bar{f}}_{L^{\infty}(B_{1})}
	\leqslant \frac{L R^{2+i(\Phi)}}{\nu_{0} K^{1+i(\Phi)}}\norm{f}_{L^{\infty}(B_{1})}
	\leqslant  \frac{L}{\nu_0}\norm{f}_{L^{\infty}(B_{1})}.
\end{align*}
By recalling $i(\Phi)> -1$ and setting 
\begin{align*}
	K:= 2\left(1+\norm{u}_{L^{\infty}(B_{1})} + \left[\frac{L}{\nu_0}\norm{f}_{L^{\infty}(B_{1})} \right]^{\frac{1}{1+i(\Phi)}}\right)
\end{align*}
and 
\begin{align*}
	R:= \varepsilon_{0}^{\frac{1}{2+i(\Phi)}},
\end{align*}
we see that $\bar{u}$ solves the equation \eqref{sr:2} in the same class as \eqref{xi-eq} under the small regime in \eqref{small}.

\subsection{Basic regularity results}
\label{sec2-3}
In this subsection, we state some basic regularity results for \eqref{xi-eq}. The first key tool to be employed later is the classical Ishii-Lions lemma, see \cite{CIL1}.

\begin{lem}[Ishii-Lions Lemma]  
	\label{lem_IL}
	Let $u$ be a viscosity solution of \eqref{xi-eq} with $\osc\limits_{B_{1}}u \leqslant 1$ and $\norm{f}_{L^{\infty}(B_{1})}\leqslant \varepsilon_{0}\ll 1$ under the assumptions \ref{a1}-\ref{a3}, where $\xi\in \R^{n}$ is any vector. Suppose that $\mathcal{B}\subset B_{1}$ is an open subset and $\psi \in C^{2}(\mathcal{B}\times\mathcal{B})$. Define a map $v : \mathcal{B}\times \mathcal{B}\rightarrow \R$ as
	\[
		v(x,y):= u(x)-u(y).
	\] 
	Suppose further $(\bar{x},\bar{y})\in \mathcal{B}\times \mathcal{B}$ is a local maximum point of $v-\psi$ in $\mathcal{B}\times\mathcal{B}$. Then, for each $\delta>0$, there exist matrices $X_{\delta},Y_{\delta}\in \mathcal{S}(n)$ such that 
	\begin{align*}
		G_{\xi}(\bar{x},D_{x}\psi(\bar{x},\bar{y}),X_{\delta}) \leqslant 0 \leqslant 
		G_{\xi}(\bar{y},-D_{y}\psi(\bar{x},\bar{y}),Y_{\delta})
	\end{align*}
	and
	\begin{align*}
		-\left( \frac{1}{\delta} + \norm{A} \right)I \leqslant
		\begin{pmatrix}
			X_{\delta} & 0 \\
			0 & -Y_{\delta}
		\end{pmatrix}
		\leqslant
		A + \delta A^{2}
	\end{align*}
	with $A:= D^{2}\psi(\bar{x},\bar{y})$. 
\end{lem}

Another important result to be applied afterwards is the results of \cite{IS2} in our settings.

\begin{thm}[Imbert-Silvestre]
	\label{thm_IS}
	Let $u\in C(B_{1})$ be a viscosity solution to \eqref{xi-eq} for some $\xi\in \R^{n}$. Suppose there exists $\gamma>0$ such that 
	\begin{enumerate}
		\item[1.] for all $(x,p)\in B_{1}\times \R^{n}$ with $|p|>\gamma$, it holds that
		\begin{align*}
			G_{\xi}(x,p,0) \leqslant c_{0}|p|
		\end{align*}	
		for some constant $c_{0}>0$ and 	
		\item[2.] for any fixed $(x,p)\in B_{1}\times \R^{n}$ with $|p|>\gamma$,  $G_{\xi}(x,p,M)$ is uniformly elliptic with respect to $M$.   
	\end{enumerate}
	Then $u\in C_{\loc}^{0,\alpha}(B_{1})$ for some $\alpha\in (0,1)$. In particular, the following estimate
	\begin{align*}
		\norm{u}_{C^{0,\alpha}(B_{1/2})} \leqslant c\norm{u}_{L^{\infty}(B_{1})}
	\end{align*}
	holds true for some constant $c>0$. The constants $\alpha\in (0,1)$ and $c>0$ depending on $n$, the ellipticity constants and the parameter $\gamma>0$.
\end{thm}


\section{H\"older continuity}
\label{sec3}

In this section we provide H\"older regularity for solutions of \eqref{xi-eq}, where $\xi$ is any vector, under the small regime.

\begin{lem}[H\"older continuity]
	\label{lem_dHC}
	Let $u$ be a viscosity solution of \eqref{xi-eq} under the assumptions \ref{a1}-\ref{a3} with $\osc\limits_{B_{1}}u \leqslant 1$, $\norm{f}_{L^{\infty}(B_{1})}\leqslant \varepsilon_{0}<1$ and $\nu_0 = \nu_1 =1$. Let $B_{R}\equiv B_{R}(x_0)\subset B_{1}$ be any ball. Then, we have the following regularity results:
	\begin{description}
		\item[(R1)\label{R1}] If $-1<i(\Phi)<0$ and $|\xi|=0$, then $u$ is Lipschitz continuous in $B_{R/2}$ with the estimate
	\begin{align}
		\label{lem_HC:1}
		[u]_{C^{0,1}(B_{R/2})} \leqslant C_{sl}
	\end{align}
for some constant $C_{sl}\equiv C_{sl}(n,\lambda,\Lambda,i(\Phi),L,R)$. 
		\item[(R2)\label{R2}] If $i(\Phi)\geqslant 0$ and $|\xi|> A_{0}$ with $A_{0}\equiv A_{0}(n,\lambda,\Lambda,i(\Phi),L,R)$, then $u$ is Lipschitz continuous in $B_{R/2}$ with the estimate 
	\begin{align}
		\label{lem_HC:2}
		[u]_{C^{0,1}(B_{R/2})} \leqslant C_{dl}
	\end{align}
for some constant $C_{dl}\equiv C_{dl}(n,\lambda,\Lambda,i(\Phi),L,R)$. 
		\item[(R3)\label{R3}] If $i(\Phi)\geqslant 0$ and $|\xi|\leqslant A_0$, then $u\in C^{0,\beta}(B_{R/2})$ with the estimate 
		\begin{align}
			\label{lem_HC:3}
			[u]_{C^{0,\beta}(B_{R/2})} \leqslant C_{ds},
		\end{align}
where $\beta\equiv \beta(n,\lambda,\Lambda, R, A_0)\in (0,1)$ and $C_{ds}\equiv C_{ds}(n,\lambda,\Lambda, R, A_0)$.
	\end{description}
\end{lem}

\begin{proof}
	For the proof of \ref{R1} and \ref{R2}, it suffices to show that there exist positive constants $L_1$ and $L_2$ such that 
	\begin{align}
		\label{HC:1}
		\mathcal{L}:= \sup\limits_{x,y\in B_{R}} \left( u(x)-u(y)-L_1\omega(|x-y|) - L_{2}\left( |x-z_0|^2 + |y-z_0|^2 \right) \right)\leqslant 0
	\end{align}
	for every $z_0\in B_{R/2}$, where 
	\begin{align}
		\label{HC:omega}
		\omega(t) = 
		\begin{cases} t-\omega_{0}t^{\frac{3}{2}} & \mbox{if } t \leqslant t_0:= \left(\frac{2}{3\omega_0}\right)^{2}, \\ \omega(t_0) & \mbox{if } t\geqslant t_0. \end{cases}
	\end{align}
We choose $\omega_0\in (0,2/3)$ in such a way that $t_0\geqslant 1$. For instance, we take any constant $\omega_0 \leqslant 1/3$. By the contradiction, suppose that there are no such positive constants $L_1$ and $L_{2}$ satisfying \eqref{HC:1} for every $z_0\in B_{R/2}$. Then there exists a point $z_0\in B_{R/2}$ so that $\mathcal{L}>0$ for all numbers $L_1>0$ and $L_2>0$. Now we define two auxiliary functions $\phi, \psi : \overline{B_{R}}\times \overline{B_{R}}\rightarrow \R$ given by 
	\begin{align}
		\label{HC:2}
		\psi(x,y):= L_1\omega(|x-y|) + L_{2}\left( |x-z_0|^2 + |y-z_0|^2 \right)
	\end{align}
	and 
	\begin{align}
		\label{HC:3}
		\phi(x,y):= u(x)-u(y)-\psi(x,y).
	\end{align}
Let $(\bar{x},\bar{y})\in \overline{B_{R}}\times \overline{B_{R}}$ be a maximum point for $\phi$. Then we have 
	\begin{align*}
	\phi(\bar{x},\bar{y}) = \mathcal{L}>0
	\end{align*}
and 
	\begin{align*}
	L_1\omega(|\bar{x}-\bar{y}|) + L_{2}\left( |\bar{x}-z_0|^2 + |\bar{y}-z_0|^2 \right)
	\leqslant 
	\osc\limits_{B_{1}}u \leqslant 1.
	\end{align*}
Now we select
	\begin{align*}
		L_2:= \frac{64}{R^2}.
	\end{align*}
	This choice of $L_2$ ensures 
	\begin{align}
		\label{HC:7}
		|\bar{x}-z_0| + |\bar{y}-z_0| \leqslant \frac{R}{4} 
		\quad\text{and}\quad
		|\bar{x}-\bar{y}|\leqslant \frac{R}{4}.
	\end{align}
	This means that the points $\bar{x}$ and $\bar{y}$ belong to the open ball $B_{R}$ and also we are able to assume that $\bar{x}\neq \bar{y}$; otherwise $\mathcal{L} \leqslant 0$ clearly. The rest of the proof is divided into several steps.
	
\textbf{Step 1.} We are in a position to apply Lemma \ref{lem_IL} in order to ensure the existence of a limiting subjet $(\xi_{\bar{x}},X_{\delta})$ of $u$ at $\bar{x}$ and a limiting superjet $(\xi_{\bar{y}},Y_{\delta})$ of $u$ at $\bar{y}$, where 
\begin{align*}
	\xi_{\bar{x}}:= D_{x}\psi(\bar{x},\bar{y})= L_{1}\omega'(|\bar{x}-\bar{y}|)\frac{\bar{x}-\bar{y}}{|\bar{x}-\bar{y}|} + 2L_{2}(\bar{x}-z_0)
\end{align*}
and 
\begin{align*}
	\xi_{\bar{y}}:= -D_{y}\psi(\bar{x},\bar{y}) = L_{1}\omega'(|\bar{x}-\bar{y}|)\frac{\bar{x}-\bar{y}}{|\bar{x}-\bar{y}|} - 2L_{2}(\bar{y}-z_0),
\end{align*}
such that matrices $X_{\delta}$ and $Y_{\delta}$ satisfy the matrix inequality
\begin{align}
	\label{HC:10}
		\begin{pmatrix}
			X_{\delta} & 0 \\
			0 & -Y_{\delta}
		\end{pmatrix}
		\leqslant
		\begin{pmatrix}
			Z & -Z \\
			-Z & Z
		\end{pmatrix}
		+ (2L_2 + \delta)I,
	\end{align}
where 
\begin{align*}
	\begin{split}
	Z&:= L_1D^2(\omega(|\cdot|))(\bar{x}-\bar{y})
	\\&
	= L_1\left[ \frac{\omega'(|\bar{x}-\bar{y}|)}{|\bar{x}-\bar{y}|}I + \left( \omega''(|\bar{x}-\bar{y}|)- \frac{\omega'(|\bar{x}-\bar{y}|)}{|\bar{x}-\bar{y}|} \right)\frac{(\bar{x}-\bar{y})\otimes (\bar{x}-\bar{y})}{|\bar{x}-\bar{y}|^{2}} \right]
	\end{split}
\end{align*}
and the constant $\delta>0$ only depends on the norm of $Z$, which can be selected sufficiently small. Applying the inequality \eqref{HC:10} for vectors of the form $(z,z)\in \R^{2n}$, we find 
\begin{align*}
	\inner{(X_{\delta}-Y_{\delta})z}{z} \leqslant (4L_2 + 2\delta)|z|^{2}.
\end{align*}
The last inequality yields that all the eigenvalues of the matrix $(X_{\delta}-Y_{\delta})$ are not larger than $4L_2 + 2\delta$. On the other hand, applying again \eqref{HC:10} for the vector $\bar{z}:= \left( \frac{\bar{x}-\bar{y}}{|\bar{x}-\bar{y}|}, \frac{\bar{y}-\bar{x}}{|\bar{x}-\bar{y}|} \right)$, we have 
\begin{align*}
	\begin{split}
	\inner{(X_{\delta}-Y_{\delta}) \frac{\bar{x}-\bar{y}}{|\bar{x}-\bar{y}|}}{ \frac{\bar{x}-\bar{y}}{|\bar{x}-\bar{y}|}}
	&\leqslant
	\left( 4L_2 + 2\delta + 4L_1\omega''(|\bar{x}-\bar{y}|)\right)\left|\frac{\bar{x}-\bar{y}}{|\bar{x}-\bar{y}|} \right|^2
	\\&
	=
	\left( 4L_2 + 2\delta - \frac{6\omega_0 L_1}{|\bar{x}-\bar{y}|^{1/2}}\right)\left|\frac{\bar{x}-\bar{y}}{|\bar{x}-\bar{y}|} \right|^2
	\\&
	\leqslant
	\left( 4L_2 + 2\delta - 6\omega_0 L_1\right)\left|\frac{\bar{x}-\bar{y}}{|\bar{x}-\bar{y}|} \right|^2,
	\end{split}
\end{align*}
where we have used the definition of $\omega$ in \eqref{HC:omega} together with  $|\bar{x}-\bar{y}|\leqslant 1/4$ in \eqref{HC:7}. So at least one eigenvalue of $(X_{\delta}-Y_{\delta})$ is not larger than $4L_2 + 2\delta -6\omega_0 L_1$, where this quantity can be negative for large values of $L_{1}$. By the definition of the extremal Pucci operator, we see 
\begin{align*}
	\begin{split}
	P_{\lambda,\Lambda}^{-}(X_{\delta}-Y_{\delta})
	&\geqslant
	-\lambda(4L_2 + 2\delta-6\omega_0 L_1)- \Lambda (n-1)(4L_2 + 2\delta)
	\\&
	\geqslant
	-(\lambda + (n-1)\Lambda)(4L_2 + 2\delta) + 6\omega_0\lambda L_1.
	\end{split}
\end{align*}
From two viscosity inequalities and the uniform ellipticity, we have 
\begin{align*}
	\Phi(\bar{x},|\xi + \xi_{\bar{x}}|)F(X_{\delta}) \leqslant f(\bar{x}),\quad 
	\Phi(\bar{y},|\xi + \xi_{\bar{y}}|)F(Y_{\delta})\geqslant f(\bar{y})
\end{align*}
 and 
\begin{align*}
	F(X_{\delta})\geqslant F(Y_{\delta}) + P_{\lambda,\Lambda}^{-}(X_{\delta}-Y_{\delta}).
\end{align*}

Combining last three displays, we have 
\begin{align}
	\label{HC:18}
	\begin{split}
		6\omega_0\lambda L_1
		&\leqslant
		(\lambda + (n-1)\Lambda)(4L_2 + 2\delta)
		\\&
		\quad +
		\frac{f(\bar{x})}{\Phi(\bar{x},|\xi + \xi_{\bar{x}}|)}
		- 
		\frac{f(\bar{y})}{\Phi(\bar{y},|\xi + \xi_{\bar{y}}|)}.
	\end{split}
\end{align}

At this stage, we shall separate it into several cases depending on the quantity of $|\xi|$ and the positiveness of $i(\Phi)$. 

\textbf{Step 2: Proof of \ref{R1}.} Suppose $-1<i(\Phi)<0$ and $\xi = 0$. By triangle inequality $\eqref{HC:7}_{2}$, we observe that 
\begin{align}
	\label{HC:19}
	|\xi_{\bar{x}}|\leqslant L_1(1+\frac{3}{2}\omega_0) + 2L_2
	\leqslant
	\frac{7}{4} L_1
\end{align}
and
\begin{align}
	\label{HC:19_1}
	|\xi_{\bar{x}}| \geqslant L_{1}\left(1-\frac{3\omega_0}{2}|\bar{x}-\bar{y}|^{\frac{1}{2}}\right)-3L_2
	\geqslant 
	\frac{3L_1}{4}-3L_2 \geqslant 3L_2
\end{align}
for all $L_1\geqslant 8L_2$. In the exactly same way, we see 
\begin{align}
	\label{HC:20}
	|\xi_{\bar{y}}|\leqslant \frac{7}{4}L_1 
	\quad\text{and}\quad
	|\xi_{\bar{y}}|\geqslant 2L_2
\end{align}
for all $L_1\geqslant 8L_2$. Then we have 
\begin{align}
	\label{HC:20_1}
	\frac{f(\bar{x})}{\Phi(\bar{x},|\xi_{\bar{x}}|)} \leqslant
	c \frac{\norm{f}_{L^{\infty}(B_{1})}}{|\xi_{\bar{x}}|^{i(\Phi)}}
	\leqslant
	\frac{c}{L_1^{i(\Phi)}}
\end{align}
and 
\begin{align}
	\label{HC:20_2}
	\frac{-f(\bar{y})}{\Phi(\bar{y},|\xi_{\bar{y}}|)} \leqslant
	c \frac{\norm{f}_{L^{\infty}(B_{1})}}{|\xi_{\bar{y}}|^{i(\Phi)}}
	\leqslant 
	 \frac{c}{L_1^{i(\Phi)}}
\end{align}
for a constant $c\equiv c(i(\Phi),L)$. Using the last two displays in \eqref{HC:18}, we obtain 
\begin{align*}
	6\omega_{0}\lambda L_1 \leqslant 
	(\lambda + (n-1)\Lambda)(4L_2 + 2\delta) + \frac{c}{L_{1}^{i(\Phi)}}
\end{align*} 
for a constant $c\equiv c(n,\lambda,\Lambda,i(\Phi),L,R)$. Recalling $-1<i(\Phi)<0$ and taking $L_1$ large enough, depending only on $n,\lambda,\Lambda,i(\Phi),L$ and $R$, we get a contradiction. Then the first part of the lemma is proved. 

\textbf{Step 3: Proof of \ref{R2}.} We suppose that $i(\Phi)\geqslant 0$ and $|\xi|>A_{0}$ for a constant $A_{0}$ to be determined in a moment. We set 
\begin{align}
	\label{HC:A0}
	A_0 := \frac{35 L_1}{2}
\end{align}
for $L_1>1$ to be selected soon. This choice of $A_0$ together with \eqref{HC:19} and \eqref{HC:20_1} leads to 
\begin{align*}
	|\xi + \xi_{\bar{x}}|\geqslant A_0 - \frac{A_0}{10} = \frac{9A_0}{10}
	\quad\text{and}\quad
	|\xi + \xi_{\bar{y}}|\geqslant  \frac{9A_0}{10}.
\end{align*} 
Therefore, we have 
\begin{align*}
	\frac{f(\bar{x})}{\Phi(\bar{x},|\xi + \xi_{\bar{x}}|)} \leqslant
	c \frac{\norm{f}_{L^{\infty}(B_{1})}}{|\xi + \xi_{\bar{x}}|^{i(\Phi)}}
	\leqslant
	\frac{c}{A_0^{i(\Phi)}}
\end{align*}
and 
\begin{align*}
	\frac{-f(\bar{y})}{\Phi(\bar{y},|\xi + \xi_{\bar{y}}|)} \leqslant
	c \frac{\norm{f}_{L^{\infty}(B_{1})}}{|\xi + \xi_{\bar{y}}|^{i(\Phi)}}
	\leqslant 
	 \frac{c}{A_{0}^{i(\Phi)}}
\end{align*}
for a constant $c\equiv c(i(\Phi),L)$. Again using the last two displays in \eqref{HC:18}, we obtain 
\begin{align*}
	6\omega_{0}\lambda L_1 \leqslant 
	(\lambda + (n-1)\Lambda)(4L_2 + 2\delta) + \frac{c}{L_{1}^{i(\Phi)}}
\end{align*} 
for a constant $c\equiv c(n,\lambda,\Lambda,i(\Phi),L,R)$. By choosing $L_1$ large enough, depending only on $n,\lambda,\Lambda,i(\Phi),L$ and $R$, we have again a contradiction. Indeed, we have proved the second part of the lemma.

\textbf{Step 4: Proof of \ref{R3}.} Finally, we shall focus on proving \ref{R3}. Suppose now $|\xi| \leqslant A_0$, where $A_0$ has been determined in \eqref{HC:A0}. We consider the operator 
\begin{align*}
	G_{\xi}(x,p,M):= \Phi(x,|\xi + p|)F(M)-f(x).
\end{align*}
In fact, $G_{\xi}(x,p,M)$ is uniformly elliptic, whenever $|p|>2A_0$. At this stage, we apply Theorem \ref{thm_IS} to conclude the last part of the Lemma. The proof is complete.

\end{proof}


\section{Approximation}
\label{sec4}
Now we prove a key approximation lemma, which plays a crucial role in later arguments.

\begin{lem}
	\label{lem:AL}
	Let $u\in C(B_{1})$ be a viscosity solution of \eqref{xi-eq} with $\osc\limits_{B_{1}} \leqslant 1$, where $\xi\in \R^{n}$ is arbitrarily given. Suppose \ref{a1}-\ref{a3} hold true for $i(\Phi)\geqslant 0$ and $\nu_0=\nu_1=1$. Then, for any $\mu>0$, there exists a constant $\delta\equiv \delta(n,\lambda,\Lambda,i(\Phi),L,\mu)$ such that if 
	\begin{align}
		\label{lem:AL:1}
		\norm{f}_{L^{\infty}(B_{1})} \leqslant \delta,
	\end{align}
	then one can find $h\in C^{1,\bar{\alpha}}(B_{3/4})$ with the estimate $\norm{h}_{C^{1,\bar{\alpha}}(B_{3/4})} \leqslant c\equiv c(n,\lambda,\Lambda)$, for some $0<\bar{\alpha}<1$, satisfying 
	\begin{align}
		\label{lem:AL:2}
		\norm{u-h}_{L^{\infty}(B_{1/2})} \leqslant \mu.
	\end{align}
\end{lem}	

\begin{proof}
	By contradiction, we suppose the conclusion of the lemma fails. Then there exist $\mu_0>0$ and sequences of $\{F_{k}\}_{k=1}^{\infty}$, $\{\Phi_{k}\}_{k=1}^{\infty}$, $\{f_{k}\}_{k=1}^{\infty}$,  and $\{u_{k}\}_{k=1}^{\infty}$ and a sequence of vectors $\{\xi_{k}\}_{k=1}^{\infty}$ such that 
	
\begin{description}	
	\item[(C1) \label{c1}] $F_{k} \in C(\mathcal{S}(n), \R)$ is uniformly $(\lambda,\Lambda)$-elliptic,
	\item[(C2) \label{c2}] $\Phi_{k} \in C( B_{1}\times [0,\infty), [0,\infty))$ such that the map $t\mapsto \frac{\Phi_{k}(x,t)}{t^{i(\Phi)}}$ is almost non-decreasing and the map $t\mapsto \frac{\Phi(x,t)}{t^{s(\Phi)}}$ is  almost non-increasing with constant $L\geqslant 1$, and $ \Phi_{k}(x,1) = 1$ for all $x\in B_{1}$,
	\item[(C3)\label{c3}] $f_{k}\in C(B_{1})$ with $\norm{f_{k}}_{L^{\infty}(B_{1})} \leqslant \frac{1}{k}$ and
	\item[(C4)\label{c4}] $u_{k}\in C(B_{1})$ with $ \osc_{B_{1}}u_k \leqslant 1$ solves the equation
	\begin{align}
		\label{AL:0}
		\Phi_{k}(x,|\xi_{k}+Du_{k}|)F_{k}(D^2u_k) = f_{k}(x),
	\end{align}
\end{description}
but 
\begin{align}
	\label{AL:1}
	\sup\limits_{x\in B_{1/2}}|u_{k}(x)-h(x)|> \mu_0
\end{align}
for all $h\in C^{1,\bar{\alpha}}(B_{3/4})$ and every $0<\bar{\alpha}<1$.

The condition \ref{c1} implies that $F_{k}$ converges to some uniformly $(\lambda,\Lambda)$-elliptic operator $F_{\infty}\in C(\mathcal{S}(n),\R)$. Applying Lemma \ref{lem_dHC}, $u_{k}\in C^{0,\beta}_{\loc}(B_{1})\cap C(B_{1})$ for some $\beta\in (0,1)$. Using \eqref{lem_HC:2}, \eqref{lem_HC:3} and Arzela-Ascoli theorem, we have that the sequence $\{u_{k}\}_{k=1}^{\infty}$ converges to a function $u_{\infty}$ locally uniformly in $B_{1}$. In particular, there holds that 
\begin{align}
	\label{AL:2}
	u_{\infty}\in C(B_{1})\quad\text{and}\quad
	\osc\limits_{B_{1}} u_{\infty} \leqslant 1.
\end{align} 

Now we prove that the limiting function $u_{\infty}$ is a viscosity solution of the homogeneous equation 
\begin{align}
	\label{AL:3}
	F_{\infty}(D^{2}u_{\infty}) = 0 \quad\text{in}\quad B_{3/4}.
\end{align}

For this, first we verify that $u_{\infty}$ is a viscosity supersolution. Let 
\begin{align*}
	p(x):= \frac{1}{2}\inner{M(x-y)}{x-y} + \inner{b}{x-y} + u_{\infty}(y)
\end{align*}
be a quadratic polynomial touching $u_{\infty}$ from below at a point $y\in B_{3/4}$. Without loss of generality, let us assume $|y|=u_{\infty}(y) = 0$. Then there exists a sequence $x_{k}\rightarrow 0$ as $k\rightarrow \infty$ such that $u_{k}-\varphi$ has a local minimum at $x_{k}$. Observe that $D\varphi(x_k)\rightarrow b$ and $D^{2}\varphi(x_k)\rightarrow M$. Since $u_{k}$ is a viscosity solution of \eqref{AL:0}, we have 
\begin{align}
	\label{AL:6}
	\Phi_{k}(x_k,|\xi_{k} + D\varphi(x_k)|)F_{k}(D^{2}\varphi(x_k)) \geqslant f_{k}(x_k).
\end{align}
For the ease of presentation, from now on we shall consider several cases depending on the boundedness of sequence $\{\xi_{k}\}_{k=1}^{\infty}$.

\textbf{Case 1: Sequence $\{\xi_{k}\}_{k=1}^{\infty}$ is unbounded.} In this case, we can assume $|\xi_{k}|\rightarrow \infty$ (up to a subsequence). As a consequence, we can show (up to a subsequence) that 
\begin{align}
	\label{AL:7}
	|\xi_{k} + D\varphi(x_k)|\leqslant |\xi_{k}|-|D\varphi(x_k)| \geqslant |\xi_{k}| - (|b|+1)\geqslant 1,
\end{align}
 which implies that
\begin{align*}
	F_{\infty}(M) 
	&= \lim\limits_{k\to\infty} F_{k}(D^2\varphi(x_k))
	\geqslant
	\lim\limits_{k\to \infty} \frac{f_{k}(x_k)}{\Phi_{k}(x_k,|\xi_{k} +D\varphi(x_k)|)}
	\notag\\&
	\geqslant
	-\lim\limits_{k\to\infty} \frac{L}{k|\xi_k + D\varphi(x_k)|^{i(\Phi)}} = 0,
\end{align*}
where we have used \ref{c2} and \eqref{AL:6}. 

\textbf{Case 2: Sequence $\{\xi_{k}\}_{k=1}^{\infty}$ bounded }
In the case we may assume $\xi_{k}\rightarrow \xi_{\infty}$ (up to a subsequence). Therefore, for the case $|\xi_{\infty}+b|\neq 0$, in the exactly same way as in \eqref{AL:7}, we infer that $F_{\infty}(M)\geqslant 0$. Then we focus on the case $|\xi_{\infty}+b|=0$. There are two possibilities as $|b|=|\xi_{\infty}| = 0$ or $b=-\xi_{\infty}$ with $|b|, |\xi_{\infty}|>0$. In those scenarios, we prove that $F_{\infty}(M)\geqslant 0$. By contradiction suppose 
\begin{align}
	\label{AL:9}
	F_{\infty}(M)<0.
\end{align}
 From the uniformly ellipticity condition of $F_{\infty}$, the matrix $M$ has at least one positive eigenvalue. Let $\R^{n}= E\oplus Q$, where $E = \text{span}\{e_1,\ldots,e_{m}\}$ is the space consisting of those eigenvectors corresponding to positive eigenvalues of $M$.

\textbf{Case 3: $b= -\xi_{\infty}$ with $|b|,|\xi_{\infty}|>0$.} Let $\gamma>0$ and set 
\begin{align*}
	p_{\gamma}(x):= p(x) + \gamma |P_{E}(x)| = \frac{1}{2}\inner{Mx}{x} + \inner{b}{x} + \gamma |P_{E}(x)|,
\end{align*}
where $P_{E}$ stands for the orthogonal projection on $E$. Since $u_{k}\rightarrow u_{\infty}$ locally uniformly in $B_{1}$ and $p(x)$ touches $u_{\infty}(x)$ from below at the origin, for $\gamma$ small enough, $p_{\gamma}(x)$ touches $u_{k}(x)$ from below at a point $x_{k}^{\gamma}\in B_{r}$ ($B_{r}$ is a small neighborhood of the origin). Moreover, there holds that $x_{k}^{\gamma}\rightarrow x_{\infty}^{\gamma}$ for some $x_{\infty}^{\gamma}$ as $k\rightarrow \infty$. At this point we consider two scenarios: $P_{E}(x_{k}^{\gamma})=0$ for all $k\in\mathbb{N}$ (up to a subsequence) or $P_{E}(x_{k}^{\gamma})\neq 0$ for all $k\in\mathbb{N}$ (up to a subsequence).

\textit{Scenario 1: $P_{E}(x_{k}^{\gamma})=0$ for all $k\in\mathbb{N}$ (up to a subsequence).} In this scenario, first we note that 
\begin{align*}
	\bar{p}_{\gamma}(x):= \frac{1}{2}\inner{Mx}{x} + \inner{b}{x} + \gamma\inner{e}{P_{E}(x)}
\end{align*}
touches $u_{k}$ from below at $x_{k}^{\gamma}$ for every $e\in \mathbb{S}^{n-1}$. A straightforward computation gives us
\begin{align*}
	D\bar{p}_{\gamma}(x_{k}^{\gamma}) = Mx_{k}^{\gamma} + b + \gamma P_{E}(e)
	\quad\text{and}\quad 
	D^2\bar{p}_{\gamma}(x_k^{\gamma}) = M.
\end{align*}
Now we select $e\in E\cap \mathbb{S}^{n-1}$ such that $P_{E}(e)=e$. Therefore, by $u_k$ being a viscosity solution of \eqref{AL:0}, we see 
\begin{align*}
	\Phi_{k}(x_{k}^{\gamma}, |\xi_{k} + Mx_{k}^{\gamma} + b + \gamma e|)F_{k}(M) \geqslant f_{k}(x_{k}^{\gamma}).
\end{align*}

We also notice that if $Mx_{\infty}^{\gamma}=0$, then for $k$ enough large, we have 
\begin{align*}
	|\xi_{k} + Mx_{k}^{\gamma} + b| \leqslant \gamma/2
	\quad\text{and}\quad 
	3\gamma/2 \geqslant |\xi_{k} + Mx_{k}^{\gamma} + b + \gamma e| \geqslant \gamma/2.
\end{align*}
Therefore, combining the last two displays and using \ref{c2} together with $\gamma\ll 1$, we have
\begin{align*}
	\begin{split}
		F_{k}(M) & \geqslant
		\frac{f_{k}(x_{k}^{\gamma})}{\Phi_{k}(x_{k}^{\gamma},|\xi_{k} + Mx_{k}^{\gamma} + b + \gamma e| )}
		\\&
		\geqslant
		\frac{-L|f_{k}(x_{k}^{\gamma})|}{ |\xi_{k} + Mx_{k}^{\gamma} + b + \gamma e|^{s(\Phi)}}
		\geqslant
		-\frac{L}{k} \left(\frac{2}{\gamma} \right)^{s(\Phi)}.
	\end{split}
\end{align*}
Letting $k\rightarrow \infty$ in the last display, we obtain $F_{\infty}(M)\geqslant 0$. In the situation $|Mx_{\infty}^{\gamma}|>0$, we first look at the subcase $E\equiv \R^{n}$ and choose $e\in \mathbb{S}^{n-1}$ such that 
\begin{align*}
	|Mx_{\infty}^{\gamma} + \gamma P_{E}(e)| = |Mx_{\infty}^{\gamma} + \gamma e|>0.
\end{align*}  
Therefore, for $k$ large enough, we have
\begin{align}
	\label{AL:16}
	|Mx_{k}^{\gamma} + \gamma e|\geqslant \frac{1}{2}|Mx_{\infty}^{\gamma} + \gamma e| > 0
	\quad\text{and}\quad
	|\xi_{k} + b| \leqslant \frac{1}{8}|Mx_{\infty}^{\gamma} + \gamma e|.
\end{align}
 On the other hand, if $E\not\equiv \R^{n}$, then we can find $e\in \mathbb{S}^{n-1} \cap E^{\perp}$ so that 
 \begin{align*}
 	|Mx_{\infty}^{\gamma} + \gamma P_{E}(e)| = |Mx_{\infty}^{\gamma}|>0.
 \end{align*}
 Again for $k$ large enough, we have 
 \begin{align}
 	\label{AL:16_1}
 	|Mx_{k}^{\gamma}|\geqslant \frac{1}{2}|Mx_{\infty}^{\gamma}|
	\quad\text{and}\quad
	|\xi_{k} + b| \leqslant \frac{1}{8}|Mx_{\infty}^{\gamma}|.
 \end{align}

  As a consequence, using either \eqref{AL:16} or \eqref{AL:16_1}, we see 
 \begin{align*}
 	|\xi_{k} + Mx_{k}^{\gamma} + b + \gamma P_{E}(e)| > \frac{1}{4}|Mx_{\infty}^{\gamma} + \gamma P_{E}(e)| > 0.
 \end{align*}
 Again applying \ref{c2} and taking into account the last display, we have 
 \begin{align*}
 	\begin{split}
 		F_{k}(M) &\geqslant 
 		\frac{f_{k}(x_{k}^{\gamma})}{\Phi_{k}(x_{k}^{\gamma},|\xi_{k} + Mx_{k}^{\gamma} + b + \gamma P_{E}(e)| )}
 		\\&
 		\geqslant
 		-\left(\frac{L}{ |\xi_{k} + Mx_{k}^{\gamma} + b + \gamma P_{E}(e)|^{i(\Phi)}} + \frac{L}{ |\xi_{k} + Mx_{k}^{\gamma} + b + \gamma P_{E}(e)|^{s(\Phi)}} \right)|f_{k}(x_{k}^{\gamma})|
 		\\&
 		\geqslant
 		\frac{-L 4^{s(\Phi)}}{k}\left(\frac{1}{|Mx_{\infty}^{\gamma} + \gamma P_{E}(e)|^{i(\Phi)}} + \frac{1}{|Mx_{\infty}^{\gamma} + \gamma P_{E}(e)|^{s(\Phi)}}  \right).
 	\end{split}
 \end{align*}
 Again letting $k\rightarrow \infty$ in the last display, we again arrive at $F_{\infty}(M) \geqslant 0$. 
 
 \textit{Scenario 2: $P_{E}(x_{k}^{\gamma})\neq 0$ for all $k\in \mathbb{N}$ (up to a subsequence).} In this scenario, we note that $P_{E}(x)$ is smooth and convex in a small neighborhood of $x_{k}^{\gamma}$. Let us denote 
 \begin{align*}
 	\zeta_{k}^{\gamma}:= \frac{P_{E}(x_{k}^{\gamma})}{|P_{E}(x_{k}^{\gamma})|}.
 \end{align*}
 A direct computation yields
 \begin{align*}
 	D(|P_{E}(\cdot)|)(x_{k}^{\gamma}) = \zeta_{k}^{\gamma}
 	\quad\text{and}\quad
 	D^{2}(P_{E}(|\cdot|))(x_{k}^{\gamma}) = \frac{1}{|P_{E}(x_k^{\gamma})|}\left( I-\zeta_{k}^{\gamma}\otimes \zeta_{k}^{\gamma} \right).
 \end{align*}
 
Hence, with $u_{k}$ being a viscosity solution of \eqref{AL:0}, we have the following viscosity inequality

\begin{align*}
	\Phi_{k}(x_{k}^{\gamma}, |\xi_{k} + Mx_{k}^{\gamma} + b + \gamma \zeta_{k}^{\gamma}|)
	F_{k}\left(M + \frac{1}{|P_{E}(x_k^{\gamma})|}\left( I-\zeta_{k}^{\gamma}\otimes \zeta_{k}^{\gamma} \right)\right) \geqslant f_{k}(x_{k}^{\gamma}).
\end{align*}
Observing that $|\zeta_{k}^{\gamma}|=1$ and letting $e:= \zeta_{k}^{\gamma}$, we can perform the same procedure as in the first scenario of $P_{E}(x_{k}^{\gamma})=0$ by considering the cases of $Mx_{\infty}^{\gamma}=0$ and $Mx_{\infty}^{\gamma}\neq 0$. Finally, we conclude that $F_{\infty}(M)\geqslant 0$ when $b=-\xi_{\infty}\neq 0$, which contradicts to \eqref{AL:9}.

\textbf{Case 4: $b=\xi_{\infty} = 0$.} In fact, this case is much easier to handle. Since $\frac{1}{2}\inner{Mx}{x}$ touches $u_{\infty}(x)$ from below at the origin and $u_{k}\rightarrow u_{\infty}$ locally uniformly, the function 
\begin{align*}
	\hat{p}_{\gamma}(x) := \frac{1}{2}\inner{Mx}{x} + \gamma |P_{E}(x)|
\end{align*}
touches $u_{k}$ from below at a point $\hat{x}_{k}^{\gamma}\in B_{r}$ ($B_{r}$ is a small neighborhood of the origin) for $\gamma > 0$ sufficiently small. Again the sequence $\{\hat{x}_{k}^{\gamma}\}$ is uniformly bounded. As in \textbf{Case 3}, we analyze those two scenarios $P_{E}(\hat{x}_{k}^{\gamma})=0$ for all $k\in \mathbb{N}$ (up to a subsequence) and $P_{E}(\hat{x}_{k}^{\gamma})\neq 0$ for all $k\in \mathbb{N}$ (up to a subsequence). All in all, we conclude $F_{\infty}(M)\geqslant 0$ in this case.
 
Finally, taking into account all cases we have analyzed above, we have shown that $u_{\infty}$ is a viscosity supersolution of \eqref{AL:3}. In order to prove that $u_{\infty}$ is a viscosity subsolution of \eqref{AL:3}, we show that $-u_{\infty}$ is a viscosity supersolution of $\hat{F}_{\infty}(D^2h) = 0$, where $\hat{F}_{\infty}(M)= -F_{\infty}(-M)$ is uniformly $(\lambda,\Lambda)$-elliptic operator as well. Therefore, $u_{\infty}$ is a viscosity solution of \eqref{AL:3}. From the regularity results of \cite[Chap. 5]{CC1}, we see $u_{\infty}\in C^{1,\bar{\alpha}}_{\loc}(B_{3/4})$ for some $\bar{\alpha}\in (0,1)$. Moreover, $\norm{u_{\infty}}_{C^{1,\bar{\alpha}}(B_{1/2})} \leqslant c\equiv c(n,\lambda,\Lambda)$ via \eqref{AL:2}. So choosing $h:= u_{\infty}$ in \eqref{AL:1}, we have a contradiction. The proof is complete.

\end{proof}


\section{Proof of Theorem \ref{thm:mthm}}
\label{sec5}
Now we provide a proof of Theorem \ref{thm:mthm}. Let $u\in C(B_{1})$ be a viscosity solution with $\osc\limits_{B_{1}}u \leqslant 1$, $\norm{f}_{L^{\infty}(B_{1})}\leqslant \delta \ll 1$ for a constant $\delta\equiv \delta(n,\lambda,\Lambda,i(\Phi),L)$ to be determined in a moment and $\nu_0 = \nu_1 = 1$. The proof is divided into two main parts, where in the first part we shall deal with the case $i(\Phi)\geqslant 0$ and the remaining case $-1<i(\Phi)<0$ will be investigated in the second part.

\textbf{Part 1: $i(\Phi)\geqslant 0$.} Let us first fix a point $y\in B_{1/2}$ and an exponent with
\begin{align}
	\label{mthm:1}
	0< \beta < \min\left\{\bar{\alpha}, \frac{1}{1+s(\Phi)} \right\}.
\end{align}

We prove that there exist universal constants $0<r \ll 1$, $C_{0}>1$ and a sequence of affine functions 
\begin{align}
	\label{mthm:2}
	l_{k}(x):= a_k + \inner{b_{k}}{x},
\end{align}
where $\{a\}_{k=1}^{\infty}\subset \R$ and $\{b_{k}\}_{k=1}^{\infty}\subset \R^{n}$, such that for every $k\in \mathbb{N}$:
\begin{description}
	\item[(E1)\label{e1}]
	$\sup\limits_{x\in B_{r^{k}}(y)} |u(x)-l_{k}(x)| \leqslant r^{k(1+\beta)}$,
	\item[(E2)\label{e2}]
	$|a_{k}-a_{k-1}| \leqslant C_{0}r^{(k-1)(1+\beta)}$ and
	\item[(E3)\label{e3}]
	$|b_{k}-b_{k-1}| \leqslant C_{0}r^{(k-1)\beta}$.
\end{description}
We show these estimates by mathematical induction. For the simplicity, we divide the proof into several steps.

\textbf{Step 1. The basis of induction.} Without loss of generality we can assume $y=0$ by translating $x\mapsto y + \frac{1}{2}x$. Let us set 
\begin{align*}
	l_1(x):= h(0) + \inner{Dh(0)}{x},
\end{align*}
where $h$ is the approximation function coming from Lemma \ref{lem:AL} for a certain constant $\mu>0$ to be determined in a few lines. Then there exists a constant $C_0\equiv C_0(n,\lambda,\Lambda)>1$ such that 
\begin{align*}
	\norm{h}_{C^{1,\bar{\alpha}}(B_{3/8})} \leqslant C_0
	\quad\text{and}\quad
	\sup\limits_{x\in B_{r}} |h(x)-l_1(x)| \leqslant C_0r^{1+\bar{\alpha}}
\end{align*} 
for every $r\leqslant 3/8$. The triangle inequality yields
\begin{align*}
	\sup\limits_{x\in B_{r}} |u(x)-l_1(x)| \leqslant \mu + C_0r^{1+\bar{\alpha}}.
\end{align*}
We first select a universal constant $0<r\ll 1$ satisfying
\begin{align}
	\label{mthm:6}
	r^{\beta} \leqslant \frac{1}{2},\quad
	C_{0}r^{1+\bar{\alpha}} \leqslant \frac{1}{2}r^{1+\beta}
	\quad\text{and}\quad
	r^{1-\beta(1+s(\Phi))} \leqslant 1,
\end{align}
which is possible by \eqref{mthm:1}. In a sequel, we select a constant $\mu>0$ as 
\begin{align}
	\label{mthm:7}
	\mu:= \frac{1}{2}r^{1+\beta},
\end{align}
which fixes an arbitrary constant $\mu>0$ in Lemma \ref{lem:AL}. In turn, there exists a constant $\delta\equiv \delta(n,\lambda,\Lambda,i(\Phi),L,\beta)$ verifying the smallness assumption $\norm{f}_{L^{\infty}(B_{1})}\leqslant \delta$, but such a smallness assumption can be assumed without loss of generality. Therefore, to conclude this step we set 
\begin{align*}
	a_0:= 0,\quad a_1:= h(0),\quad b_0 = 0\quad\text{and}\quad
	b_1:= Dh(0).
\end{align*}
These choices with \eqref{mthm:6} and \eqref{mthm:7} verify that the estimates \ref{e1}-\ref{e3} are satisfied for $k=1$.

\textbf{Step 2: Induction process.} Now we suppose that the hypotheses of the induction have been established for $k=1,2,\ldots,m$ for $m\geqslant 1$. We show that the estimates \ref{e1}-\ref{e3} hold true for $k=m+1$. For this, we introduce an auxiliary function as 
\begin{align*}
	w_m(x):= \frac{u(r^{m}x)- l_{m}(r^{m}x)}{r^{m(1+\beta)}}.
\end{align*}
We note that $w_{m}$ solves the following equation in the viscosity sense
\begin{align*}
	\Phi_{m}(x, |r^{-m\beta}b_{m} + Dw_{m}|)F_{m}(D^{2}w_{m}) = f_{m}(x),
\end{align*}
where 
\begin{align*}
	F_{m}(M):= r^{m(1-\beta)}F(r^{(\beta-1)m}M),
\end{align*}
which is uniformly $(\lambda,\Lambda)$-operator, the function
\begin{align*}
	\Phi_{m}(x,t):= \frac{\Phi(r^{m}x,r^{m\beta}t)}{\Phi(r^{m}x,r^{m\beta})}\quad (x\in B_{1}, t>0)
\end{align*}
still satisfies the properties that the map $t\mapsto \frac{\Phi_{m}(x,t)}{t^{i(\Phi)}}$ is almost non-decreasing, the map $t\mapsto \frac{\Phi_{m}(x,t)}{t^{s(\Phi)}}$
is almost non-increasing with the same constant $L\geqslant 1$ and $\Phi_{m}(x,1)=1$ for all $x\in B_{1}$, and
\begin{align*}
	f_{m}(x):= \frac{r^{m(1-\beta)} f(r^{m}x)}{\Phi(r^{m}x,r^{m\beta})}.
\end{align*}
Using \ref{a2} and \eqref{mthm:1}, we notice that 
\begin{align*}
	\norm{f_{m}}_{L^{\infty}(B_{1})} \leqslant
	\frac{Lr^{m(1-\beta)}\norm{f}_{L^{\infty}(B_{1})}}{ r^{m\beta s(\Phi)}}
	\leqslant
	L\delta r^{m(1-(1+s(\Phi))\beta)} \leqslant L\delta.
\end{align*}

Therefore, we are in a position to apply Lemma \ref{lem:AL} to $w_{m}$. In turn, there exists a function $\bar{h}\in C^{1,\bar{\alpha}}(B_{3/4})$  such that 
\begin{align*}
	\sup\limits_{x\in B_{r}} |w_{m}(x)-\bar{h}(x)| \leqslant \mu.
\end{align*}
Arguing as in \textbf{Step 1}, we show that 
\begin{align*}
	\sup\limits_{x\in B_{r}} |w_{m}(x)-\bar{l}(x)| \leqslant r^{1+\beta},
\end{align*}
where 
\begin{align*}
	\bar{l}(x):= \bar{a} + \inner{\bar{b}}{x}\quad \text{for some}\quad \bar{a}\in \R \text{ and } \bar{b}\in \R^{n}.
\end{align*}
Denoting 
\begin{align*}
	l_{m+1}:= l_{m}(x) + r^{m(1+\beta)}\bar{l}(r^{-m}x),
\end{align*}
we see 
\begin{align*}
	\sup\limits_{x\in B_{r^{m+1}}} |u(x)-l_{m+1}(x)|\leqslant
	r^{(m+1)(1+\beta)} 
\end{align*}
and 
\begin{align*}
	|a_{m+1}-a_{m}| + r^{m}|b_{m+1}-b_{m}| \leqslant C_0r^{m(1+\beta)}.
\end{align*}
Therefore, the $(m+1)$-th step of the induction is complete.

\textbf{Step 3: Conclusion.} Once we have the existence of universal constants $0<r\ll 1$, $C_0>1$ and a sequence of affine functions in \eqref{mthm:2} verifying the estimates \ref{e1}-\ref{e3}, the remaining part of the proof is very standard, see for instance \cite{IS1,De1}. Therefore, the proof of \eqref{thm:mthm:2} is complete when $i(\Phi)\geqslant 0$.

\textbf{Part 2: $-1<i(\Phi)<0$.} Now we shall with the case of $-1< i(\Phi)<0$. Again we fix a point $y\in B_{1/2}$. Without loss of generality, we may assume $y=0$ by using the translation $x\mapsto y + \frac{1}{2}x$. Now we apply \ref{R1} of Lemma \ref{lem_dHC} in order to ensure that 
\begin{align}
	\label{mthm:21}
	[u]_{C^{0,1}(B_{3/4})} \leqslant C_{sl}
\end{align} 
for a constant $C_{sl}\equiv C_{sl}(n,\lambda,\Lambda,i(\Phi),L)$. Therefore, it can be seen that $u$ is a viscosity solution of the equation 
\begin{align*}
	\tilde{\Phi}(x,|Dv|)F(D^2v) = \tilde{f}(x)\quad
	\text{in}\quad B_{3/4},
\end{align*}
where 
\begin{align*}
	\tilde{\Phi}(x,t):= t^{-i(\Phi)}\Phi(x,t)\quad (x\in B_{1}, t>0),
\end{align*}
which satisfies the properties that the map $t\mapsto \tilde{\Phi}(x,t)$ is almost non-increasing, the map $t\mapsto \frac{\tilde{\Phi}(x,t)}{t^{s(\Phi)-i(\Phi)}}$ is almost non-increasing with constant $L\geqslant 1$, $\tilde{\Phi}(x,1)=1$ for all $x\in B_{1}$, and 
\begin{align*}
	\tilde{f}(x) = |Du(x)|^{-i(\Phi)}f(x).
\end{align*}
Using the estimate \eqref{mthm:21} together with $\norm{f}_{L^{\infty}B_{1}} \leqslant \delta \ll 1$, we see 
\begin{align*}
	\norm{\tilde{f}}_{L^{\infty}(B_{3/4})} \leqslant
	C_{sl}^{-i(\Phi)}\delta.
\end{align*}
So we are able to apply \textbf{Part 1} of the proof in order to have \ref{e1}-\ref{e3}. This means that we have the estimate \eqref{thm:mthm:2} for $-1<i(\Phi)<0$. The proof is complete.

\vspace{0.2cm}

%


\bibliographystyle{amsplain}

\end{document}